\documentclass{amsart}

\newcommand{\M}{\ensuremath{\mathbb{M}}}
\newcommand{\A}[1][n]{\ensuremath{\M_{#1}}}

\newcommand{\F}{\ensuremath{\mathbb{F}}}
\newcommand{\Z}{\ensuremath{\mathbb{Z}}}

\DeclareMathOperator{\tr}{Tr}

\theoremstyle{plain}
\newtheorem{theorem}{Theorem}

\newtheorem{lemma}[theorem]{Lemma}
\newtheorem{proposition}[theorem]{Proposition}
\theoremstyle{definition}

\newtheorem{remark}[theorem]{Remark}

\begin{document}
\title[Matrix multiplication is determined by orthogonality and trace]{Matrix multiplication is determined\\by orthogonality and trace}
\author{Chris Heunen and Clare Horsman}
\address{Department of Computer Science, University of Oxford}
\email{\{chris.heunen,clare.horsman\}@cs.ox.ac.uk}
\begin{abstract}
	Any associative bilinear multiplication on the set of $n$-by-$n$ matrices over some field of characteristic not two, that makes the same vectors orthogonal and has the same trace as ordinary matrix multiplication, must be ordinary matrix multiplication or its opposite.
\end{abstract}
\maketitle

\section{Introduction}

Matrix multiplication is the fundamental operation of linear algebra, modelling composition of linear maps in terms of coordinates.
This article characterises matrix multiplication simply in terms of orthogonality and trace, in the following way.
Suppose there were another way to combine two matrices $x,y \in \A$ into a new one $x \star y \in \A$,
that resembles composition of linear maps in the sense that it is associative, bilinear, and respects the identity matrix:
\begin{enumerate}
  \item[(A)] ``\emph{Associativity}'': $x \star (y \star z) = (x \star y) \star z$ for all $x,y,z \in \A$;
  \item[(B)] ``\emph{Bilinearity}'': $(\lambda x) \star y = \lambda(x \star y) = x \star (\lambda y)$, $(x+y) \star z = (x \star z) + (y \star z)$, and $x \star (y+z) = (x \star y) + (x \star z)$ for scalars $\lambda$ and $x,y,z \in \A$;
  \item[(I)] ``\emph{Identity}'': $x \star 1 = x$ for all $x \in \A$, where $1 \in \A$ is the identity matrix.
\end{enumerate}
To these basic requirements we add two properties that, at first sight, do not fix much information about $x \star y$.
First, we require $x \star y$ to have the same trace as ordinary matrix multiplication $xy$:
\begin{enumerate}
  \item[(T)] ``\emph{Trace}'': $\tr(x \star y) = \tr(xy)$ for all $x,y \in \A$.
\end{enumerate}
Second, we require that $\star$ makes the same vectors orthogonal as ordinary matrix multiplication. Formulated algebraically:
\begin{enumerate}
  \item[(O)] ``\emph{Orthogonality}'': $x \star y = 0$ when $xy=yx=0$, $xx=x$, and $yy=y$, \\for $x,y \in \A$ that have rank one.
\end{enumerate}
These assumptions already imply that $x \star y$ must equal the ordinary matrix multiplication, or its opposite, after all. That is, we will prove the following theorem, for any scalar field $\F$ of characteristic not two.

\begin{theorem}\label{thm:nogo}
  A function $\star \colon \A \times \A \to \A$ satisfies (A), (B), (I), (T) and (O) if and only if either $x \star y=xy$ for all $x,y \in \A$, or $x \star y = yx$ for all $x,y \in \A$.    
\end{theorem}

This fits in the recent programme of results surrounding linear preserver problems, \textit{i.e.}~linear maps that preserve zero products, commutativity, etc.~\cite{bresarsemrl:commutativity,molnar:preservers,alaminosbresarextremeravillena:idempotents,wanglige:idempotent}. Indeed, we will rely on one of those results~\cite{alaminosbresarextremeravillena:idempotents}.
Our original motivation came from quantum theory, where the above properties arise as desiderata for a possible noncommutative extension of Bayesian inference~\cite{leiferspekkens:bayesian,barrettetal:nogo}. 


\section{The main result}

\begin{proposition}\label{prop:idempotents}
  A map $\star \colon \A \times \A \to \A$ meets (B), (I), and (O) if and only if 
  \begin{equation}\label{eq:star}
    x \star y = xy + g(xy-yx)
  \end{equation}
  for some linear map $g \colon \A \to \A$ and all $x,y \in \A$. 
\end{proposition}
\begin{proof}
  Assume (B), (I) and (O). Then \cite[Theorem~2.2]{alaminosbresarextremeravillena:idempotents} applies, giving linear maps $f,g \colon \A \to \A$ with $x \star y = f(xy)+g(yx)$. Furthermore, $x = x \star 1 = f(x 1) + g(1 x)$, and so $f(x)=x-g(x)$. Therefore $x \star y = f(xy)+g(yx) = xy+g(yx-xy)$. Taking the negative of $g$ now gives~\eqref{eq:star}.

  If $\star$ is of the form~\eqref{eq:star}, then it is easy to show that (B), (I), and (O) hold.
\end{proof}

\begin{lemma}\label{lem:trace}
  Suppose $\star$ satisfies (B), (I), and (O).
  Then (T) holds if and only if $g$ sends traceless matrices to traceless matrices.
\end{lemma}
\begin{proof}
  Property~(T) is equivalent to $\tr(g(xy-yx))=0$ for all $x,y \in \A$. 
  But matrices of the form $xy-yx$ are precisely those with trace zero~\cite{albertmuckenhoupt:tracezero}.
\end{proof}

Write $e_{ij}$ for the standard matrix units, 
so that
$e_{ij}e_{kl} = \delta_{jk} e_{il}$.
Because of bilinearity, property~(A) holds precisely when it holds for $x=e_{ab}$, $y=e_{cd}$, $z=e_{ef}$ for all $a,b,c,d,e,f \in \{1,\ldots,n\}$.
Via Proposition~\ref{prop:idempotents}, property~(A) comes down to
\begin{align}\label{eq:magicformula}
  \begin{split}
  & -\delta_{de}\delta_{af} g(e_{cb}) + \delta_{de} e_{ab} \star g(e_{cf}) - \delta_{cf} e_{ab} \star g(e_{ed}) \\
  =&
-\delta_{af}\delta_{bc} g(e_{ed}) + \delta_{bc} g(e_{ad}) \star e_{ef} - \delta_{ad} g(e_{cb}) \star e_{ef}
  \end{split}
\end{align}
for all $a,b,c,d,e,f \in \{1,\ldots,n\}$. We will use this formula very often below. 

By linearity of $g$, we may write $g(e_{ij}) = \sum_{k,l=1}^n G_{kl,ij} e_{lk}$ for entries $G_{kl,ij} \in \F$. By convention, we will write $g(ij)_{kl}$ for $G_{kl,ij}$, and $g(ii-jj)_{kl}$ for $G_{kl,ii}-G_{kl,jj}$.


\begin{lemma}\label{lem:cases}
  For distinct $i,j,k,l \in \{1,\ldots,n\}$:
  \begin{align}
    g(ij)_{kl} = 0, \label{eq:case1} \\
    g(ii-jj)_{kl} = 0. \label{eq:case2}
  \end{align}
  For distinct $i,j,k \in \{1,\ldots,n\}$:
  \begin{align}
    g(ij)_{jk} = 0, \label{eq:case3} \\
    g(ij)_{kj} = 0, \label{eq:case4} \\
    g(ij)_{ik} = 0, \label{eq:case5} \\
    g(ij)_{ki} = 0, \label{eq:case6} \\
    g(ij)_{kk} = 0, \label{eq:case7} \\
    g(ii-jj)_{kk} = 0, \label{eq:case8} \\
    g(ii-jj)_{ik} = 0, \label{eq:case9} \\
    g(ii-jj)_{ki} = 0, \label{eq:case10} \\
    g(ii-jj)_{kj} = 0, \label{eq:case11} \\
    g(ii-jj)_{jk} = 0. \label{eq:case12} 
  \end{align}
  For distinct $i,j \in \{1,\ldots,n\}$:
  \begin{align}
    g(ij)_{jj} = 0, \label{eq:case13} \\
    g(ij)_{ii} = 0, \label{eq:case14} \\
    g(ij)_{ij} = 0, \label{eq:case15} \\
    g(ii-jj)_{ij} = 0, \label{eq:case16} \\
    g(ii-jj)_{ji} = 0. \label{eq:case17}
  \end{align}
\end{lemma}
\begin{proof}
  All these equations are derived following the same pattern. 
  Let $i,j,k,l$ be distinct. Taking $a=k$, $b=l$, $c=i$, and $d=e=f=j$ in~\eqref{eq:magicformula} results in $e_{kl} \star g(e_{ij}) = 0$. Taking the trace of both sides and using property~(T) now shows $g(ij)_{kl} = \tr(g(e_{ij})e_{kl}) = \tr(e_{kl} g(e_{ij})) = \tr(e_{kl} \star g(e_{ij})) = 0$, establishing~\eqref{eq:case1}.

  Similarly, taking $a=k$, $b=l$, $c=f=i$, and $d=e=j$ in~\eqref{eq:magicformula} results in $e_{kl} \star g(e_{ii}-e_{jj}) = 0$, from which~\eqref{eq:case2} follows by taking the trace of both sides.

  For~\eqref{eq:case3}--\eqref{eq:case17} we simply list the appropriate choice of indices:
  \begin{enumerate}
  	\item[\eqref{eq:case3}:] $c=d=e=i$, $a=f=j$, $b=k$;
  	\item[\eqref{eq:case4}:] $c=d=e=i$, $b=f=j$, $a=k$;
  	\item[\eqref{eq:case5}:] $a=c=i$, $d=e=f=j$, $b=k$;
  	\item[\eqref{eq:case6}:] $a=f=i$, $d=j$, $b=c=e=k$;
  	\item[\eqref{eq:case7}:] $c=d=e=i$, $f=j$, $a=b=k$;
  	\item[\eqref{eq:case8}:] $c=f=i$, $d=e=j$, $a=b=k$;
  	\item[\eqref{eq:case9}:] $a=c=f=i$, $d=e=j$, $b=k$;
  	\item[\eqref{eq:case10}:] $b=c=f=i$, $d=e=j$, $a=k$; use~\eqref{eq:case3};
  	\item[\eqref{eq:case11}:] $c=f=i$, $b=d=e=j$, $a=k$;
  	\item[\eqref{eq:case12}:] $c=f=i$, $a=d=e=j$, $b=k$; use~\eqref{eq:case6};
  	\item[\eqref{eq:case13}:] $c=d=e=i$, $a=b=f=j$;
  	\item[\eqref{eq:case14}:] $a=b=c=i$, $d=e=f=j$; use~\eqref{eq:case13};
  	\item[\eqref{eq:case15}:] $a=c=i$, $b=d=e=f=j$;
  	\item[\eqref{eq:case16}:] $a=c=f=i$, $b=d=e=j$;
  	\item[\eqref{eq:case17}:] $b=c=f=i$, $a=d=e=j$.\qedhere
  \end{enumerate}
\end{proof}

 \begin{lemma}\label{lem:distinction}
  There is $\lambda \in \{0,-1\}$ with $g(ii-jj)_{ii}=-g(ii-jj)_{jj}=g(ij)_{ji}=\lambda$ for distinct $i,j \in \{1,\ldots,n\}$.
  Moreover, there is $z \in \A$ with $g(ii)=\lambda e_{ii}+z$ for all $i$.
\end{lemma}
\begin{proof}
  For any $k,l \in \{1,\ldots,n\}$, the entry $g(ij)_{kl}$ can only be nonzero when $k=j$ and $l=i$ by~\eqref{eq:case1}, \eqref{eq:case3}--\eqref{eq:case7} and~\eqref{eq:case13}--\eqref{eq:case15}. Let $\lambda_{ij} \in \F$ be that entry: $g(ij)=\lambda_{ij}e_{ij}$.

  Taking $b=c=e=f=i$ and $a=d=j$ in~\eqref{eq:magicformula} leads to $g(ii-jj)_{ii} = g(ij)_{ji}$.
  Similarly, $a=d=i$ and $b=c=e=f=j$ lead to $g(jj-ii)_{jj} = g(ji)_{ij}$.
  Next, $b=c=d=e=i$ and $a=f=j$ show that $g(ij)_{ji} = g(ji)_{ij}$.
  Therefore 
  \[
    g(ii-jj)_{ii}=g(ij)_{ji}=\lambda_{ij},
    \qquad
    g(ii-jj)_{jj}=-g(ij)_{ji}=-\lambda_{ij}.
  \]
  Combining this with~\eqref{eq:case2}, \eqref{eq:case8}--\eqref{eq:case12}, \eqref{eq:case16} and~\eqref{eq:case17} shows $g(ii-jj) = \lambda_{ij}(e_{ii}-e_{jj})$.

  We may write $g(ii)=\lambda_i e_{ii} +z_i$ for $\lambda_i \in \F$, and $z_i \in \A$ linearly independent from $e_{ii}$. 
  Then $\lambda_{ij}(e_{ii}-e_{jj}) = g(ii-jj) = \lambda_i e_{ii} - \lambda_j e_{jj} + (z_i-z_j)$. 
  It follows that $z_i$ does not depend on $i$, and we simply write $z$ instead.
  Similarly, $\lambda_{ij}=\lambda_i=\lambda_j$ does not depend on $i$ or $j$, and we may simply write $\lambda$ for $\lambda_{ij}$. Hence $g(ii)=\lambda e_{ii}+z$.

  It follows from the choice of indices for~\eqref{eq:case16} that $g(ij) = e_{ij} \star g(ii-jj)$, and so
  \begin{align*}
    g(ij) 
    = \lambda e_{ij} \star (e_{ii}-e_{jj}) 
    = \lambda ( -g(ij) -e_{ij} - g(ij)) 
    = -2\lambda g(ij) - \lambda e_{ij}.
  \end{align*}
  Hence $(1+2\lambda) g(ij) = -\lambda e_{ij}$. But $g(ij)=\lambda e_{ij}$ by definition, so $\lambda \in \{0,-1\}$.
\end{proof}

\begin{proof}[Proof of Theorem~\ref{thm:nogo}]
  Lemma~\ref{lem:distinction} gives $\lambda \in \{0,-1\}$ and $z \in \A$ with $g(ii) = \lambda e_{ii} + z$. 
  Let $x \in \A$, say $x=\sum_{i,j=1}^n \chi_{ij} e_{ij}$ for $\chi_{ij} \in \F$.
  It follows from Lemma~\ref{lem:cases} that
  \begin{align*}
    g(x)
     = \sum_{i,j=1}^n \chi_{ij} g(ij) 
    & = \sum_{i=1}^n \chi_{ii} g(ii) + \sum_{i \neq j} \chi_{ij} g(ij) \\
    & = \sum_{i=1}^n (\lambda \chi_{ii} e_{ii} + \chi_{ii} z) + \sum_{i \neq j} \lambda \chi_{ij} e_{ij} \\
    & = \lambda \sum_{i,j=1}^n \chi_{ij} e_{ij} + \sum_{i=1}^n \chi_{ii} z \\
    & = \lambda x + \tr(x)z.
  \end{align*}
  Therefore $x \star y = xy + \lambda(xy-yx)$. 
\end{proof}

\begin{remark}\label{rem:characteristictwo}
  The restriction of the base field $\F$ to characteristic not two is necessary. For example, if $\F=\Z_2$, then
  \begin{equation}\label{eq:characteristictwo}
    \begin{pmatrix} \alpha & \beta \\ \gamma & \varepsilon \end{pmatrix}
    \star
    \begin{pmatrix} \zeta & \eta \\ \theta & \iota \end{pmatrix}
    = 
    \begin{pmatrix} 
      \alpha \zeta + \gamma \eta & \beta \zeta + \varepsilon \eta \\
      \alpha \theta + \gamma \iota & \beta \theta + \varepsilon \iota
    \end{pmatrix}
  \end{equation}
  satisfies (A), (B), (I), (O), and (T), but differs from matrix multiplication or its opposite.
  In fact, apart from $g=0,1$, this is the only one of the form~\eqref{eq:star}, for \textit{e.g.}
  \[
    g \begin{pmatrix} \alpha & \beta \\ \gamma & \varepsilon \end{pmatrix}
    = \begin{pmatrix} \alpha & \beta \\ \gamma & \alpha \end{pmatrix}.
  \]
\end{remark}

\section*{Acknowledgements}

The physical considerations leading to Theorem~\ref{thm:nogo} were originally raised by Jonathan Barrett and Matthew Pusey, and are discussed in~\cite{barrettetal:nogo}. Chris Heunen was supported by the Office of Naval Research under grant N000141010357. Clare Horsman was supported by the CHIST-ERA DIQIP project, and by the FQXi Large Grant ``Time and the Structure of Quantum Theory''.

\bibliographystyle{plain}
\bibliography{matrixmultiplication}

\end{document}